\documentclass{article}

\usepackage{amsthm}
\usepackage{amssymb,amsmath}
\usepackage{epsfig}
\usepackage{color}
\usepackage{hyperref}
\usepackage{verbatim}

\title{Constructing thin subgroups commensurable with the figure-eight knot group
}\author{
S. Ballas  \& D. D. Long \thanks{The authors acknowledge support from U.S. National Science Foundation grants DMS 1107452, 1107263, 1107367 ``RNMS: GEometric structures And Representation
varieties'' (the GEAR Network)}}

\def\diag{\rm{diag}}

\def\tr{tr}
\def\SL{SL}
\def\SO{SO}
\def\SU{SU}

\def\qed{ $\sqcup\!\!\!\!\sqcap$}

\def\R{{\bf R}}
\def\RP{{\bf RP}}
\def\C{{\bf C}}

\newtheorem{theorem}{Theorem}[section]

\newtheorem{lemma}[theorem]{Lemma}
\newtheorem{corollary}[theorem]{Corollary}

\begin{document}
\maketitle
\begin{abstract}
 In this paper we find infinitely many lattices in $\SL(4,\R)$ each of
which contains thin subgroups commensurable with the figure-eight knot group.
\end{abstract}

\section{Introduction.}

Let $\Gamma$ be a lattice in a semi-simple Lie group $G$. Then following Sarnak (see \cite{Sa}), one says that a subgroup $\Delta$ of $\Gamma$ is  {\em thin} if $\Delta$ has
infinite index in $\Gamma$, and is Zariski dense in $G$.   Since it is rather easy to exhibit Zariski dense subgroups of 
lattices that are free products, the case of most interest is that the thin group $\Delta$ is finitely generated and 
does not decompose as a free product.

In this note, we shall exhibit subgroups of the fundamental group of the figure eight knot as subgroups  of infinitely many
incommensurable lattices. A precise statement will be given shortly.

%
%
%
%

We begin by briefly describing where the lattices in question arise. They are constructed from Theorem 6.55 of \cite{Wi} by 
a rather general construction which involves $L$, a real quadratic extension of ${\bf Q}$ and $D$,  a central simple division 
algebra of degree $d$ over $L$. However in our situation we may assume that $D=L$, and we state only this special case.

In this paper, the field $L$ is a real quadratic extension of $\bf Q$ and  if $A\in \SL(4,L)$, we will denote by 
$A^*$ the matrix obtained by taking the transpose of the matrix  obtained from $A$ by applying $\tau$ (the non-trivial 
Galois automorphism)  to all its entries. Then one has:
\begin{theorem}
\label{lattice}
Suppose that $L$ is a real quadratic extension of $\bf Q$, with Galois automorphism $\tau$. Suppose that $b_1,....,b_4$
are nonzero elements of $\bf Z$. Setting $J=\diag(b_1,...,b_4)$, then
the group 
$$\SU(J , {\cal O}_L, \tau) =  \{ A \in \SL(4, {\cal O}_L) \;\;\;|\;\;\; A^*JA=J \}$$ 
is a lattice in $\SL(4, {\bf R})$.
\end{theorem}
We also note from \cite{Wi} Proposition 6.55 that, in the case being considered here (when $D=L$), the corresponding forms
will represent zero non-trivially, and so the lattices produced are non-uniform.\\[\baselineskip]
We shall show:
%
%
%
%
\begin{theorem}
\label{main_intro}
Let $\Gamma$ denote the fundamental group of the figure eight knot and $L_d = {\bf Q}(\sqrt{d})$ for
$d$ a positive, square free integer.

Then for every such $d$, there is a subgroup of finite index in $\Gamma$, $H_d$, and
a faithful, Zariski dense representation into a lattice:
$$ r_d : H_d \longrightarrow \SU(J_d , {\cal O}_{L_d}, \tau_d) $$
\end{theorem}
The power of this statement lies in the fact that $r_d$ is faithful and Zariksi dense; much of the work of the paper is devoted to this aspect.
Once one knows this, it follows immediately that the image group $r_d(H_d)$ has infinite index in the
lattice $\SU(J_d , {\cal O}_{L_d}, \tau_d) $, since deep results of Margulis (see Chapter $13$ of  \cite{Wi}) imply that such lattices
do not have subgroups which admit homomorphisms onto $\bf Z$. Moroever, subgroups of finite index in $\Gamma$ 
are freely indecomposable, so that it follows from \ref{main_intro} that the image is thin. 

This paper is organized as follows. The starting point is a pair of representations of $\Gamma$ constructed independently by Ballas \cite{Bal2}
(which we denote $\phi_t$) and Thistlethwaite (private communication) (which we denote $\rho_v$).  These representations are conjugate when $v=2t$ and 
while one could work  with just one of them, each presents sufficiently interesting features that it seems worthwhile to include both.  
Explicit generators for both representations are included in Appendix \ref{matrices}. We begin in \S \ref{construction} with the algebraic considerations needed to construct the representations
$r_d$ of Theorem \ref{main_intro} from the Ballas-Thistlethwaite representation.  The main work here is proving  that the
representation has an integral character in an appropriate sense; the remaining ingredients in constructing $r_d$ are fairly standard once
this has been proved. 

The results of \S \ref{faithful} and \S \ref{zariski} are geometric and lie deeper. Specifically, \S \ref{faithful} is devoted to the proof that the representations $r_d$ are faithful using results from real projective geometry.

In \S \ref{zariski} we show that, with the exception of the representation corresponding to the complete hyperbolic structure, $\phi_t$ has Zariski dense image in $\SL(4,{\bf R})$.

\section{Constructing $r_d$.}
\label{construction}
This section is devoted to  outlining the computations necessary to exhibit the representations $r_d$ of Theorem \ref{main_intro}. With a view to a point that arises in the sequel (see Theorem \ref{limit_is_properlyconvex}), we begin by proving the following:
%
%
%
%
%
%
\begin{theorem}
\label{irreducible}
The representations $\rho_v$ are absolutely irreducible for real $v > 0$.
\end{theorem}
{\bf Proof.} If there is an invariant subspace of dimension one, then any commutator must have eigenvalue $=1$.
One can check that the commutator $[x^2, y]$ has such an eigenvalue only when $v$ is a root of $v^2+v+1$.
The same considerations applied to the transpose representation show there can be no invariant subspace of $\rho_v$ 
of dimension three.

The case of a two dimensional invariant subspace is somewhat more subtle. One computes that the characteristic
polynomial of the longitude is $(-Q + v)^3 (-1 + Q v^3)/v^3$, so that if there were such an invariant subspace, either
$\rho_v(\lambda)$ or $\rho_v^T(\lambda)$ would have the $1/v^3$ eigenvalue appearing in that subspace. It follows
that the
orbit of that eigenvector would be two dimensional. However, one can compute that this happens for  neither of these representations. \qed

Our next claim concerns the traces of the $\rho_v$. 
%
%
%
%
\begin{lemma}
The representation of the figure eight knot given by $\rho_v$ has traces lying in ${\bf Z}[v, 1/v]$.
\end{lemma}
\noindent
{\bf Proof.} This involves some computation, as it must. We indicate the mathematics behind the idea, with a sample of
its implementation at \cite{web1}.

Regarding $v$ as a real transcendental,  using Burnside's theorem \cite[p.\ 648 Cor.\ 3.4]{Lang}, or by inspection, one can find elements $g_1, ... .., g_{16} \in \Gamma$ 
which are a basis for the vector space of $4 \times 4$ matrices $M(4,{\bf R})$; we always choose $g_1$ to be the identity matrix, which
we denote by $I$.
 
Let $g_ 1^*, ... .., g_{16}^*$  be the dual basis with respect to trace, i.e.  $tr (g_i.g_j^* ) = \delta_ {ij}$. 
One can use the action of the figure eight knot group on this dual basis by left multiplication to obtain a $16$-dimensional representation 
of the group, i.e. if $\gamma \in \Gamma$, then its action is defined by
$$ \gamma.g^*_i = \sum_j \alpha_{ij}(\gamma) g^*_j$$
Taking traces in this equation, we get
$$ tr(\gamma.g^*_i) = \sum_j \alpha_{ij}(\gamma) tr(g^*_j)$$
Notice that since we have chosen $g_1 = I$, we have that $tr(g^*_j) = tr(g_1.g^*_j) =  \delta_ {1j}$, in particular these are all rational
integers. Writing  $I = \sum \tau_j  g^*_j$, notice that $\tau_i = tr(g_i)$ by duality; one verifies that these traces only have denominators
which are powers of $V$.

 Moreover,  multiplying by $\gamma$ and taking traces, we have 
$$tr(\gamma) = \sum_j tr(g_j) tr(\gamma.g^*_j) = \sum_{j,k} tr(g_j) tr(g^*_k) \alpha_{jk}(\gamma) =  \sum_{j} tr(g_j)  \alpha_{j1}(\gamma) $$

The upshot of these two computations is the following: If with some choice of basis, one computes that the denominators of the entries for the 
associated $16 \times 16$ regular representation, then this collection of denominators contains the denominators for the traces of the original collection of
matrices of $\Gamma$.

Therefore, if one could find a basis for which this construction gave ${\bf Z}[v, 1/v]$-matrices $(\alpha_{ij}(\gamma))$, then this would prove the result claimed
by the lemma. However, this appears to be hard. We bypass this difficulty by constructing two representations via two different choices of basis $\{g_i\}$. 
 For the first choice one sees all the matrix entries have denominators $v$, $(-1 + v)$ and  $(1 + v)$, for the second, one sees denominators
 $v$, $(1 + 3 v + v^2)$ and 
 $(1 + 3v + 4 v^2)$. Since traces are not dependent on choice of basis, the denominators of the trace of the original representation of $\Gamma$ must lie in the intersection of
 these two sets, i.e. powers of $v$. \qed
 \begin{corollary}
 If one specializes $v$ to be a unit in any number field, then the resulting representation has integral trace.
 \end{corollary}
 Since when one sets $v=2t$ the representation $\rho_v$ is conjugate to the representation $\phi_t$ and this representation has entries lying in ${\bf Q}(t)={\bf Q}(v/2)$,
we are now in a position to apply the following lemma with $k={\bf Q}(v/2)$.
%
%
%
%
\begin{theorem}
\label{finite_index_integral}
Suppose that $G < \SL(4, k)$ is a finitely generated group with the 
property that  $\tr(\gamma) \in {\cal O}_k$ for every $\gamma \in G$.

Then $G$ has a subgroup of finite index contained in $\SL(4, {\cal O}_k)$.
\end{theorem}
\noindent
{\bf Proof.}  Consider

$${\cal O}G = \{\Sigma \; a_i \gamma_i \;\;|\;\;a_i\in {\cal O}_k, \gamma_i \in G \;\;\}$$
where the sums are finite.  It  is shown in \cite{Bass} (see Proposition 2.2 and Corollary 2.3), that ${\cal O}G$ is 
an order of a central simple subalgebra $B \subset M(4,k)$ defined over $k$. 
Now while ${\cal O}G$ need 
not be an order in $M(4,k)$, it is known that it is contained
in some maximal order ${\cal D}$ of $M(4,k)$ (cf. \cite{Rei} p. 131, Exercise 5 
and the proof of Lemma 2.3 of \cite{LRT}).

Now it is a standard fact that the groups of elements of norm $1$ in
orders contained in $M(4,k)$ are commensurable (since the intersection of two orders is an order and the 
unit groups of orders will be irreducible lattices in $\SL(4,{\bf R})\times \SL(4,{\bf R})$, see \cite{Wi} Chapter 15I).
In particular, $\SL(4, {\cal O}_k)$ and ${\cal D}^{1}$  are commensurable.
Let $\Delta = \SL(4, {\cal O}_k) \cap {\cal D}^{1}$,  which has finite index in both groups. Then 
$G \leq {\cal D}^{1}$, so that $G \cap \Delta$ has  finite index in $G$ and lies
inside $\SL(4, {\cal O}_k)$ as required. \qed\\[\baselineskip]
What have achieved to this point is that when one specializes $v$ to be any unit of a number field $k$, there is a representation,
namely $\phi_{v/2}$, of  a subgroup of finite index in $\Gamma$ whose image has entries lying in ${\cal O}_k$.

We now turn our attention to the unitary aspects required by Theorem \ref{lattice}.  There is an obvious involution on the matrices
in the image of $\rho_v$ given by transposing and mapping $v \rightarrow 1/v$. We denote this operation by $A \rightarrow A^*$. 
A routine computation on the generators reveals:
\begin{lemma}
\label{unitary_form}
There is a matrix $Q_v$ with $det(Q_v) \neq 0$ for which 
$$ A^*.Q_v.A= Q_v$$
for all $A \in \phi_{v/2} (\Gamma)$.
\end{lemma}

Since it's required by the result of Theorem \ref{lattice}, for the rest of this paper we assume $k =  {\bf Q}(\sqrt{d})$ is a real 
quadratic number field equipped with involution $\tau_d$. Specialising $v$ to be a unit of this field with the property that $\tau_d(v)=1/v$
(e.g. one can take the square of a random unit of ${\bf Q}(\sqrt{d})$; we call these the  {\em positive units}), we see that Lemma \ref{unitary_form} and Theorem \ref{finite_index_integral}
taken together prove:
\begin{theorem} 
\label{lattice_construction}
For each positive square free integer $d$, and a positive unit $u$  in the ring of integers of  ${\bf Q}(\sqrt{d})$, there is a representation 
of a subgroup of finite index $H_{d,u}$ in $\Gamma$
$$ r_{d,u}: H_{d,u} \longrightarrow SU( Q_v, {\cal O}_k , \tau_d)$$
\end{theorem}

There is one final consideration that must be addressed, namely the form described by Witte is diagonal and $Q_v$
is not. However this concern is addressed as follows. It can easily be shown by using the Gram-Schmidt process for
example, that there is a change of basis matrix $M_d \in GL( {\bf Q}(\sqrt{d}))$ for which $M_d^*.Q_v.M_d$ is a diagonal
form $\Delta_v$. Then it is a standard argument (see for example Lemma 2.2 of \cite{ALR}) that  the groups $SU( Q_v, {\cal O}_k , \tau_d)$ and 
$SU( \Delta_v, {\cal O}_k , \tau_d)$ are commensurable, so at possibly the expense of passing to a further subgroup of
finite index, we obtain a representation 
$$ r_{d,u}: H_{d,u} \longrightarrow SU( \Delta_v, {\cal O}_k , \tau_d)$$
as required by Theorem \ref{lattice}.\\[\baselineskip]
%
%
%

%
%
%
%
%
%

\section{Projective considerations: $r_d$ is faithful.}
\label{faithful}
This section contains the proof that one can find many representations $r_d$ which are faithful. It is basically geometric in
nature and relies upon the fact that these representations are associated to convex real projective structures on the figure eight knot.

Throughout we denote the complement of the figure-eight knot in ${\bf S}^3$ by $M$ and denote its fundamental group by $\Gamma$. We begin with some considerations of a fairly general nature. We recall that a group is said to be {\em non-radical} if there is
no infinite normal nilpotent subgroup. The figure eight knot group (and indeed any finite volume hyperbolic manifold group)
is non-radical. The following is a classical theorem of Zassenhaus.
\begin{theorem} (See Kapovich \cite[Thm 8.4]{Kap})
Suppose that $G$ is a finitely generated non-radical group and $\cal L$ any linear Lie group
and that  $\{ \sigma_n \}$ is a convergent sequence of discrete, faithful representations of $G$ into $\cal L$,
say $\sigma_n \rightarrow \sigma_\infty$.

Then $\sigma_\infty$ is discrete and faithful.
\end{theorem}
We will also need the following:
%
%
\begin{theorem} (Goldman-Choi \cite{GC})
\label{limit_is_properlyconvex}
Let $G$ be a finitely generated non-radical group.

Suppose that $\Omega_n$ is a sequence of properly convex open domains in $\RP^d$ and
$\sigma_n : G \rightarrow SL(d+1, {\bf R})$ a convergent sequence of discrete faithful representations for which
$\sigma_n(G) < Aut(\Omega_n)$. Denote the limit representation by $\sigma_\infty$.

Then if $\sigma_\infty$ is irreducible, it preserves some properly convex open subset of $\RP^d$.
\end{theorem}

We define a subset of the representation variety    $ \Omega(G ; \RP^d) < Hom( G , SL(d+1, {\bf R}))$ to be the set of representations
$ \sigma : G \rightarrow  SL(d+1, {\bf R})$ satisfying
\begin{itemize}
\item $\sigma$ is discrete and faithful
\item There is  a properly convex, open domain $\Omega_\sigma \subset \RP^d$ for which $\sigma(G) < Aut(\Omega_\sigma)$
\end{itemize}
%
%
%
%
%
In the case that $G$ is the fundamental group of e.g. a hyperbolic $d$-manifold, the set $ \Omega(G ; \RP^d)$ is nonempty, since
it contains the representation corresponding to the complete structure $\rho_1 = \phi_{\frac{1}{2}}$.
\begin{theorem}
\label{components}
 The  path component of the set ${\cal G} = \{ \; v \;\;|\;\; \rho_v \in \Omega(\Gamma ; \RP^3) \}$ which contains  $1$ is open and closed.
\end{theorem}
\noindent
In particular, once Theorem \ref{components} is proven, we have that ${\cal G}$ is some interval, and therefore ${\cal G} = (0,\infty)$ since
the only place the representations $\rho_v$ fail to be defined is at $v=0$. This implies that any specialization of $v$ in $(0,\infty)$
is discrete and faithful, so that taken in conjunction with Theorem \ref{lattice_construction}, we will have proved all of
Theorem \ref{lattice}, barring the fact the image is Zariski dense.\\[\baselineskip]
{\bf Proof of Theorem \ref{components}.} The fact that the set is closed follows from a concatenation of results proved above: Theorem \ref{faithful}
implies that the endpoint of a path of discrete faithful representations is discrete and faithful, then Theorem \ref{limit_is_properlyconvex} implies 
the resulting representation  is the holonomy of a properly convex structure, since we proved in \ref{irreducible} that the limit 
representation is irreducible.

Openness is a good deal more subtle. By applying a theorem of Cooper and Long (\cite[Thm 0.4]{CL}) we find that a small deformation of the holonomy of the complete hyperbolic structure of $M$ whose restriction to some (hence any) peripheral subgroup has a common fixed point in $\RP^3$ will itself be the holonomy of a projective structure on $M$. Furthermore, the cusp of this projective structure is foliated by projective rays with a common endpoint. 

A priori, this projective structure need not be properly convex. However the main result of \cite{CLT2} shows that if the cusp of the resulting projective structure satisfies slightly stronger hypotheses then the deformed projective structure will be properly convex. Roughly speaking, the additional hypothesis is that the cusp of the deformed projective structure must admit a second foliation by ``strictly convex'' hypersurfaces that is transverse to the previously mentioned foliation by projective rays.  

In \cite{Bal} it is shown that for all $t\in (0,\infty)$ the representation $\rho_t$ satisfy the hypotheses of the previously mentioned theorems. Since the two families of representations are conjugate this implies that for all $u\in (0,\infty)$, $\rho_u$ also satisfy the same hypotheses. As a result we find that for $u$ sufficiently close to $1$ that $\rho_u$ is the holonomy of a properly convex projective structure on $M$. The holonomy of such a structure is necessarily discrete and faithful and so for $u$ sufficiently close to $1$ we find that $\rho_u\in \Omega(\Gamma;\RP^3)$. We thus conclude that $\mathcal{G}$ is open.  \qed



\section{Zariski Denseness}\label{zariski}

In this section we analyze the Zariski closure of the groups $\phi_t(\Gamma)$ (which we denote $G_t$) proving in particular that away from the
complete representation, this Zariski closure
is all of  $ \SL(n,{\bf R})$. We adopt a largely geometric point of view; other approaches are possible, see the remarks at the conclusion of this section.

We begin by with some background and results which can be found in Benoist \cite{Ben}. A comprehensive summary of the necessary 
background in algebraic groups, notions of proximality, and Zariski closures can be found in \cite{Mar1,Mar2}. 
We begin by defining \emph{proximality} in the context of groups, group actions, and representations. In all cases proximality is related 
to the existence of unique attracting fixed points. 

Let $G$ be a subgroup of $\SL(n,{\bf R})$ and let $g$ be an element of $G$. We say that $g$ is \emph{proximal} if it has a unique eigenvalue of largest modulus. In this case it is easy to see that this eigenvalue is real. If in addition this eigenvalue of largest modulus is positive then we say that $g$ is \emph{positive proximal}. A group is called \emph{proximal} if it contains a proximal element and \emph{positive proximal} if {\bf every} proximal element is positive proximal. 

A group action of $G$ on ${\bf RP}^{n-1}$ is \emph{proximal} if for any pair of points $x,y\in {\bf RP}^{n-1}$ there is a sequence $\{g_m\}$ of elements in $G$ such that 
$$\lim_{m\to \infty}g_m\cdot x=\lim_{m\to\infty}g_m\cdot y.$$

If $G$ is a connected semi-simple Lie group then a representation $\rho:G\to \SL(n,{\bf R})$ is \emph{proximal} if the weight space corresponding to highest restricted weight is 1-dimensional. More specifically, if we let $G=KAN$ be an Iwasawa decomposition of $G$, where $K$ is a maximal compact subgroup, $A$ is a maximal abelian subgroup, and $N$ is a maximal nilpotent subgroup, then the set 
$$\{x\in {\bf R}^{n}\mid n(x)=x\ \forall n\in N\}$$
is a line. 

We now discuss some relations between these notions. We say that a group $\Gamma\subset \SL(n,{\bf R})$ is \emph{strongly irreducible} if every finite index subgroup of $\Gamma$ is irreducible. The following two theorems relate the various notions of proximality. 

\begin{theorem}[{\cite[Thm 2.9]{GG}}]\label{proximal1}
Let $G$ be a subgroup of $\SL(n,{\bf R})$ then the following are equivalent
\begin{enumerate}
 \item $G$ is strongly irreducible and proximal.
 \item $G$ is irreducible and its action on ${\bf RP}^{n-1}$ is proximal.
\end{enumerate}
\end{theorem}

\begin{theorem}[{\cite[Thm 6.3]{AMS}}]\label{proximal2}
 Let $G$ be a semi-simple Lie group with finite center and let $\rho:G\to \SL(n,{\bf R})$ be an irreducible representation. Then the following are equivalent. 
 \begin{enumerate}
  \item $\rho$ is proximal.
  \item $\rho(G)$ is a proximal. 
 \end{enumerate}
\end{theorem}

Given a strongly irreducible and proximal subgroup $G\subset \SL(n,{\bf R})$ we can define the \emph{limit set} of $G$, which we denote by $\Lambda_G$ as the closure of the set of attractive fixed points in $\RP^{n-1}$ of the proximal elements of $G$. In \cite[Thm 2.3]{GG} it is shown that the action of $G$ on $\Lambda_G$ is minimal (i.e.\ any non-empty closed $G$-invariant subset of $\RP^{n-1}$ contains $\Lambda_G$). 

When $G$ is a Zariski closed semi-simple subgroup of $\SL(n,\R)$ we can describe $\Lambda_G$ more explicitly. Since $G$ is proximal, $\Lambda_G$ is non-empty and so let $x\in \Lambda_G$. Consider the orbit $G\cdot x$ of $x$. First, observe that $G\cdot x\subset \Lambda_G$ by $G$-invariance. Furthermore, since $G\cdot x$ is the orbit of an algebraic group acting algebraically on a variety we have that $G\cdot x$ is a smooth subvariety of $\RP^{n-1}$. If we let $H$ be the Zariski closure of $G\cdot x$ then we see that $G\cdot x$ is open in $H$ and so $H\backslash (G\cdot x)$ is a Zariski closed subset of $\RP^{n-1}$. Since Zariski closed sets are closed in the standard topology, minimality implies that $H\backslash( G\cdot x)$ is either empty or $\Lambda_G$. However, $x\notin H\backslash( G\cdot x)$ and so $H\backslash (G\cdot x)$ is empty. As a result we see that $G\cdot x$ is a non-empty closed $G$-invariant subset and thus $G\cdot x=\Lambda_G$. By minimality of the action on the limit set we also find that 
$\Lambda_G$ is the \emph{unique} closed orbit for the action of $G$ on $\RP^{n-1}$.  

We can say even more. Let $G=KAN$ be an Iwasawa decomposition and let $x_N$ be the (unique) point in $\RP^{n-1}$ corresponding to a highest weight vector. By our choice of $x_N$ we see that the groups $A$ and $N$ both fix $x_N$ and so $G\cdot x_N=K\cdot x_N$. Thus the orbit of $x_N$ is closed and so $G\cdot x_N=\Lambda_G$ (i.e.\ the limit set of $G$ is the orbit of a compact group.)

We can now identify the Zariski closure of $\phi_t(\Gamma)$, which we denote $G_t$. The representation $\phi_{1/2}$ corresponds to the holonomy of the complete finite volume hyperbolic structure on the figure-eight knot and so by the Borel density theorem we have $G_{1/2}=\SO(3,1)$. The main goal of this section is to show that when $t\neq 1/2$ that $G_t=\SL(4,{\bf R})$. The proof is based on the following heuristic. 
$$\textit{A Lie subgroup of $\SL(4,{\bf R})$ with a large orbit in ${\bf RP}^3$ must be large.}$$

\begin{theorem}
If $t\neq 1/2$ then $G_t=\SL(4,{\bf R})$. In particular, for $t\neq 1/2$, $\phi_t(\Gamma)$ is Zariski dense.  
\end{theorem}

\begin{proof}
 Let $t\neq 1/2$. We have previously seen that $\rho_t$ is an absolutely irreducible representation whose image preserves a properly convex open subset of $\RP^3$. We first show that $\rho_t(\Gamma)$ is strongly irreducible. Suppose for contradiction that $\rho_t(\Gamma)$ is not strongly irreducible. By work of Benoist \cite[Lem 2.9]{Ben} we can find a finite index subgroup of $\Gamma$ that splits as a non-trivial direct sum. However, any finite index subgroup of $\Gamma$ is the fundamental group of a finite volume hyperbolic 3-manifold and such groups never admit non-trivial direct sum decompositions.
 
 Let $G^0_t$ be the connected component of $G_t$ containing the identity. We claim that $G^0_t$ is semi-simple. To see this observe that by \cite[Lem 2.6]{GG} we find that $G^0_t$ acts irreducibly on $\R^4$. As a result the action of $G^0_t$ turns $\R^4$ into a simple $\R[G^0_t]$ module. Let $R_t$ be the unipotent radical of $G^0_t$. Since $R_t$ is unipotent and solvable the Lie-Kolchin theorem implies that there is a non-trivial $\C[R_t]$-submodule, $E_\C$, of similtaneous 1-eigenvectors of $R_t$ in ${\bf C}^4$. A simple computation shows that $E_\C$ is conjugation invariant and so there is a non-trivial $\R[R_t]$-submodule, $E_\R\subset \R^4$, of similtaneous 1-eigenvectors of $R_t$ whose complexification is $E_\C$. However, since $R_t$ is normal in $G^0_t$ we see that $E_\R$ is a non-trivial $\R[G^0_t]$-submodule. However, simplicity implies that this submodule must be all of ${\bf R}^4$. Therefore $R_t$ acts trivially on ${\bf R}^4$ and is thus trivial. We conclude that $G^0_t$ is reductive. Furthermore, 
since $G^0_t$ is proximal and irreducible we see that it has trivial center and is thus semi-simple.

 
 The group $\phi_t(\Gamma)$ is proximal and so Theorem \ref{proximal2} implies that the representation induced by the inclusion of $G^0_t$ into $\SL(4,\R)$ is also proximal.
 
Next we show that $\Lambda_{G^0_t}$ contains a codimension-1 submanifold. Let $\Gamma_p$ be a peripheral subgroup of $\Gamma$. By work of \cite[\S 6]{Bal} it is possible to conjugate so that $\phi_t(\Gamma_p)$ to a lattice in the 2-dimensional abelian Lie group $H$ of matrices of the form 
$$\begin{pmatrix}
   1 & 0 & s & s^2/2-t\\
   0 & e^t & 0 & 0\\
   0 & 0 & 1 & s\\
   0 & 0 & 0 & 1
  \end{pmatrix}
$$
Thus we see that the Zariski closure of $\phi_t(\Gamma_p)$ (and hence $G_t$) contains $H$. A generic orbit of $H$ can be written in homogenous coordinates as
$$\{[-\log(\left|x\right|)+y^2/2+c:\epsilon x:y:1]\mid x>0\},$$
where $\epsilon\in\{\pm 1\}$ and $c\in \R$ (see Figure \ref{domain}). Since $G^0_t$ is irreducible we see that $\Lambda_{G^0_t}$ contains a point, $z$ of one of these orbits and so $H\cdot z\subset G\cdot z\subset \Lambda_{G^0_t}$. Furthermore, since $\Lambda_{G^0_t}$ is closed we see that it contains the closure of this orbit, which is the boundary of a properly convex domain $\Omega\subset \RP^3$. 

\begin{figure}
 \begin{center}
  \includegraphics[scale=.6]{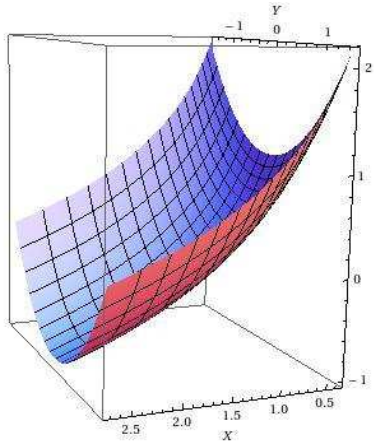}
  \caption{\label{domain} A generic orbit of $H$}
 \end{center}
\end{figure}

There are now two cases to consider: either $\Lambda_{G^0_t}$ has codimension 0 and is thus equal to $\RP^3$ or $\Lambda_{G^0_t}$ is a codimension 1 submanifold of $\RP^3$ and is thus equal to $\partial \Omega$. In the first case work of Benoist \cite[Lem 3.9 \& Cor 3.5]{Ben} shows that $G^0_t=\SL(4,\R)$, in which case we are done.

We now rule out the second case. In this case we see that there is a point $w\in \RP^3$ such that $\Lambda_{G^0_t}=K\cdot w=\partial \Omega$, where $K$ is a maximal compact subgroup of $G^0_t$. As a result we see that (up to a subgroup of index 2) $K$ preserves $\Omega$. Since $K$ is compact we see that there is a point $x_0\in \Omega$  such that both $x_0$ and its dual point $x_0\in \Omega^\ast$\footnote{see \cite[\S 9]{Mar1} for a definition of dual point of properly convex domains} are both fixed by $K$. The hyperplane dual to $x_0^\ast$ provides us an affine patch containing $\Omega$ and for which $x_0$ is the center of mass of $\Omega$. In these affine coordinates $K$ acts by affine transformations fixing $x_0$. 

If we let $\mathcal{E}$ be a John ellipsoid\footnote{A John ellipsoid for a convex subset $\Omega$ of affine space is an ellipsoid of maximal Euclidean volume contained in $\Omega$ with the same center of mass} for $\Omega$ centered at $x_0$ in this affine patch, then we see that $K$ preserves $\mathcal{E}$. Since $\mathcal{E}$ is has maximal volume $\partial \mathcal{E}$ has nonempty intersection with $\partial \Omega$. Since $K$ acts transitively on $\partial \Omega$ we see that $\partial \Omega \subset \partial \mathcal{E}$. 

For dimensional reasons we see that $\partial \Omega$ has non-empty interior in $\partial \mathcal{E}$. Since $\partial \Omega$ is a $K$-orbit we see that it is an open subset of $\partial \mathcal{E}$. We conclude that $\partial \Omega=\partial \mathcal{E}$ and thus $\Omega$ is an ellipsoid. Since $\phi_t(\Gamma)$ preserves this ellipsoid we see that it is conjugate to a subgroup of $\SO(3,1)$. This is a contradiction since, for example the image of the longitude has an eigenvalue whose inverse is not also an eigenvalue.   
\end{proof}
\noindent
{\bf Remark.}
Given the explicit nature of the subgroups in this paper there are alternative, more algebraic ways of proving Zariski denseness. For example, it suffices to show 
that the adjoint action of $\rho_t(\Gamma)$ on $\sl_4$ is irreducible (a Mathematica notebook with irreducibility calculations for our groups can be found at \cite{web2}). 
The anonymous referee also pointed out other techniques that could be employed to show that such explicit subgroups are Zariski dense (see \cite{Riv} and 
\cite[Thm 9.10]{PraRap}). We chose to use the above proof because of its geometric nature which highlights the relationship between convex projective structures 
and the Zariski closure of their holonomy representations. We plan to pursue this relationship in more detail in future work.

\appendix
\section{The matrices}\label{matrices}

%
%
%
%
%
%
%
%
%
 The discrete faithful representation corresponding to the hyperbolic structure occurs at $v=1$.
 $$\rho_v(x)=  \left( \begin{array}{cccc}
\frac{3}{2} &  \frac{1}{2}  &    \frac{1}{2}(1+1/v) &    \frac{1}{\sqrt{12}} (1 - 1/v)\\ 
-\frac{1}{2}  & \frac{1}{2}  & -   \frac{1}{2}(1+1/v)  & - \frac{1}{\sqrt{12}} (1 - 1/v)\\ 
1 & 1 &1 & 0\\
0& 0 &0 &1 
\end{array}\right)$$

$$\rho_v(y)=  \left( \begin{array}{cccc}
v +  \frac{1}{2} & -v +  \frac{1}{2}   &   \frac{1}{2}  &    \frac{1}{\sqrt{12}} (7-4v)\\ 
  \frac{1}{2}  &  \frac{1}{2}   &   \frac{1}{2}  &   \frac{1}{\sqrt{12}} (4/v-1)\\ 
 \frac{1}{2} v& -\frac{1}{2} v  &  1 &    \frac{1}{\sqrt{3}}(1-v)\\ 
  \frac{\sqrt{3}}{2}v &   -\frac{\sqrt{3}}{2}v  &  0 &  2 - v 
\end{array}\right)$$

%
%
%
%
The discrete faithful representation corresponding to the hyperbolic structure occurs at $t=1/2$.
$$\phi_t(x)=  \left( \begin{array}{cccc}
1 & 0 &1 & t - 1 \\
  0 & 1 &1 & t \\
 0 & 0 &1 & t + \frac{1}{2} \\
 0 & 0 &0 & 1\\
\end{array}\right)$$

$$\phi_t(y)=  \left( \begin{array}{cccc}
1 & 0 &0 & 0 \\
  2+ 1/t & 1 &0 & 0 \\
 2 & 1 &1 & 0\\
 1 & 1 &0 & 1\\
\end{array}\right)$$
The representations become conjugate for $v=2t$.

%
%
%
%

%
%
%
Department of Mathematics,\\ University of California,\\ Santa Barbara, 
CA 93106.\\
\noindent Email:~long@math.ucsb.edu
\\

\noindent Department of Mathematics,\\ University of California,\\ Santa Barbara, 
CA 93106.\\
\noindent Email:~sballas@math.ucsb.edu\\[\baselineskip]

\end{document}